\documentclass[11pt]{article}

\usepackage{amsmath, amssymb, mathrsfs, amsthm, amsfonts}
\usepackage[margin=3cm]{geometry}
\usepackage[utf8]{inputenc}

\newtheorem{thm}{Theorem}[section]   
\newtheorem{defn}[thm]{Definition}
\newtheorem{lem}[thm]{Lemma}
\newtheorem{ex}{Example}

\def\NN{{\mathbb N}}

\def\B{{\mathcal B}}
\def\S{{\mathcal S}}
\def\M{{\mathcal M}}

\title{On Hyperoctahedral Enumeration System, Application to Signed Permutations}
\author{Iharantsoa Vero RAHARINIRINA\footnote{University of Antananarivo, Madagascar and University of Caen Normandy, France,
		\texttt{ihvero@yahoo.fr}}}

\begin{document}

\maketitle

\begin{abstract}
In this paper, we start by giving the definitions and basic facts about hyperoctahedral number system.  There is a natural correspondence between the integers expressed in the latter and the elements of the hyperoctahedral group when we use the inversion statistic on this group to code the signed permutations.  We use this correspondence to define a way with which the signed permutations group can be ordered. With this classification scheme, we can find the $ r $-th signed permutation from a given number $ r $ and vice versa without consulting the list  in lexicographical order of the elements of the signed permutations group. 
\end{abstract}
\section{Introduction}
Let us denote by:
\begin{itemize}
\item $ \NN $ the set of non negative integers including $ 0 $,
\item $ [n] $ the set $ \{1,\cdots ,n\} $,
\item $ [\pm n] $ the set $ \{-n,\cdots,-1,1,\cdots ,n\} $,
\item $ \mathcal{S}_n$ the symmetric group of degree $ n $
\item $ \{e_1,e_2,\ldots,e_n\} $ the set of the standard basis vectors of the vector space $ \mathbb{R}^n $.
\end{itemize}
We represent an element $ \sigma $ of $ \mathcal{S}_n$ as the word
$ \sigma_1 \cdots \sigma_n $ where $ \sigma_i =\sigma(i) $.
\begin{defn}
	A bijection $ \pi :[\pm n]\longrightarrow [\pm n] $ satisfying $ \pi(-i)=-\pi(i) $ for all $ i $ in $ [\pm n] $ is called "signed permutation".
\end{defn}
We also write a signed permutation $ \pi $ in the form
\begin{equation*}
	\pi = \left( \begin{array}{cccc} 1 & 2 & \dots & n\\ 
\varepsilon_1 \sigma_1 & \varepsilon_2 \sigma_2 & \dots & \varepsilon_n\sigma_n \end{array} \right) \; 
\text{ with } \sigma\in \mathcal{S}_n \text{ and } \varepsilon_i \in \{\pm 1\}\; .
\end{equation*}
Under the ordinary composition of mappings, all signed permutations of the elements of $ [n] $ form a group $ \B_n $ called hyperoctahedral group of rank $ n $.

Laisant seems to be the first to classify permutations in lexicographic order. He associates bijectively a rank to each element of the symmetric group $ \S _n $. The key to this association is based on the existence of the inversion statistic on $ \S _n $.
\begin{defn}
	Let $ \sigma \in \S _n $.
A pair of indices $ (i,j) $ with $ i<j $ and $ \sigma _i >\sigma _j $ is called an inversion of $ \sigma $.
\end{defn}
In this work, we study an analogue of this classification for signed permutations. One method of studying the proposed ordering is using the new statistic number of $ i $-inversions on $ \B _n $. To define this last, it is convenient to see $ \B _n $ as the Coxeter group with root system
$$\Phi_n = \{\pm e_i,\,\pm e_i \pm e_{j} \ |\ 1 \leq i \neq j \leq n\},$$ 
and positive root system $$\Phi_n^+ = \{e_{k},\, e_i + e_{j},\, e_i - e_{j}\ |\ k \in [n],\, 1 \leq i < j \leq n\}.$$
Let us consider the following subset of $\Phi_n^+$ defined by
\begin{equation} \label{eq1}
\Phi_{n,i}^+ = \{ e_i,\, e_i + e_{j},\, e_i - e_{j}\ |\ i < j \leq n\} .
\end{equation}
\begin{defn}
	We define the number of $i$-inversions of $\pi \in \B _n$ by 
\begin{equation} \label{eq2}
\mathtt{inv}_i\,\pi = \# \{v \in \Phi_{n,i}^+ \ |\ \pi^{-1}(v) \in -\Phi_n^+\}.
\end{equation} 
\end{defn}\begin{ex}
	In table \ref{B2}, we see the corresponding $1$-inversions and $2$-inversions of the eight elements of $\B_2$:
\end{ex}
\begin{table}[h]
	\begin{center}
		\begin{tabular}{ c | c }
			Signed permutation $ \pi $& $\mathtt{inv}_1\pi\ \mathtt{inv}_2\pi $\\ \hline
			& \\
			$ \left( \begin{array}{cc} \mathbf{1} & \mathbf{2} \\ \mathbf{1} & \mathbf{2} \end{array} \right)$ &
			$  0 \ 0$ \\
			$ \left( \begin{array}{cc} \mathbf{1} & \mathbf{2} \\ \mathbf{1} & -\mathbf{2} \end{array} \right)$ &
			$ 0 \ 1$ \\
			$ \left( \begin{array}{cc} \mathbf{1} & \mathbf{2} \\ \mathbf{2} & \mathbf{1} \end{array} \right)$ &
			$ 1 \ 0$ \\
			$\left( \begin{array}{cc} \mathbf{1} & \mathbf{2} \\ \mathbf{2} & -\mathbf{1} \end{array} \right)$ &
			$ 1 \ 1$ \\
			$ \left( \begin{array}{cc} \mathbf{1} & \mathbf{2} \\ -\mathbf{2} & \mathbf{1} \end{array} \right)$ &
			$ 2 \ 0$ \\
			$ \left( \begin{array}{cc} \mathbf{1} & \mathbf{2} \\ -\mathbf{2} & -\mathbf{1} \end{array} \right)$ &
			$ 2 \ 1$ \\
			$ \left( \begin{array}{cc} \mathbf{1} & \mathbf{2} \\ -\mathbf{1} & \mathbf{2} \end{array} \right)$ &
			$ 3 \ 0$ \\
			$ \left( \begin{array}{cc} \mathbf{1} & \mathbf{2} \\ -\mathbf{1} & -\mathbf{2} \end{array} \right)$ &
			$ 3 \ 1$
		\end{tabular}
	\end{center}
	\caption{The $1$-inversions and $2$-inversions of the elements of $\B_2$.} \label{B2table}
\label{B2}	
\end{table}
The easiest way to explain lexicographic ordering is with an example.
\begin{ex}
	The set of all permutations of order three in lexicographic order is :
	$$
	abc,acb,bac,bca,cab,cba .
	$$
\end{ex}
The main purpose of this work is to give a classification of signed permutations of elements of $ [n] $ when the elements of the set $ [\pm n] $ are ordered as follow : $ 1,2,\cdots ,n,-n,\cdots,-1 $. For instance, we give in table \ref{table3} the classification of elements of $ \B_3 $.
 This article is organised as follow : in section \ref{sectionHyperS} we define the hyperoctahedral enumeration system and give some properties of the numbers in this system, then in section  \ref{sectionClassification} we write how to classify the elements of an hyperoctahedral group, it is one application of the hyperoctahedral system.
\begin{table}[h]
	\begin{center}
		\begin{tabular}{ c | c | c || c | c | c }
			Rank & $ \pi=\pi_1\pi_2\pi_3$  & $\mathtt{inv}_1\pi:\mathtt{inv}_2\pi: \mathtt{inv}_3\pi $& Rank & $ \pi=\pi_1\pi_2\pi_3$  & $\mathtt{inv}_1\pi:\mathtt{inv}_2\pi: \mathtt{inv}_3\pi $\\ \hline
1& 1 2 3 & 0:0:0  & 	25 & -3 1 2& 3:0:0  \\ \hline
2 & 1 2-3& 0:0:1 & 	26 & -3 1-2& 3:0:1 \\ \hline
3 & 1 3 2& 0:1:0   &	27 & -3 2 1& 3:1:0 \\ \hline
4 & 1 3-2& 0:1:1   &	28 & -3 2-1& 3:1:1  \\ \hline
5 & 1-3 2& 0:2:0   &	29 & -3-2 1& 3:2:0  \\ \hline
6 & 1-3-2& 0:2:1   &	30 & -3-2-1& 3:2:1  \\ \hline
7 & 1-2 3& 0:3:0 &	31 & -3-1 2& 3:3:0     \\ \hline
8 & 1-2-3& 0:3:1  &	 32 & -3-1-2& 3:3:1  \\ \hline
9 & 2 1 3& 1:0:0 &  33 & -2 1 3& 4:0:0   \\ \hline
10 & 2 1-3& 1:0:1  & 34 & -2 1-3& 4:0:1    \\ \hline
11 & 2 3 1& 1:1:0  & 35 & -2 3 1& 4:1:0    \\ \hline
12 & 2 3-1& 1:1:1  & 36 & -2 3-1& 4:1:1    \\ \hline
13 & 2-3 1& 1:2:0 & 37 & -2-3 1& 4:2:0    \\ \hline
14 & 2-3-1& 1:2:1  &	38 & -2-3-1& 4:2:1  \\ \hline
15 & 2-1 3& 1:3:0  & 39 & -2-1 3& 4:3:0   \\ \hline
16 & 2-1-3& 1:3:1 & 40 & -2-1-3& 4:3:1   \\ \hline
17 & 3 1 2& 2:0:0& 41 & -1 2 3& 5:0:0  \\ \hline
18 & 3 1-2& 2:0:1 & 42 & -1 2-3& 5:0:1 \\ \hline
19 & 3 2 1& 2:1:0 & 43 & -1 3 2& 5:1:0  \\ \hline
20 & 3 2-1& 2:1:1 & 44 & -1 3-2& 5:1:1  \\ \hline
21 & 3-2 1& 2:2:0 & 45 & -1-3 2& 5:2:0  \\ \hline
22 & 3-2-1& 2:2:1 &  46 & -1-3-2& 5:2:1  \\ \hline
23 & 3-1 2& 2:3:0  & 47 & -1-2 3& 5:3:0  \\ \hline
24 & 3-1-2& 2:3:1 & 48 & -1-2-3& 5:3:1  \\ \hline
		\end{tabular}
	\end{center}		
	\caption{Elements of $ \B_3 $ with the indication of their rank and their corresponding number in the hyperoctahedral system.}
	\label{table3}
\end{table}
\section{Hyperoctahedral System}\label{sectionHyperS}
Let us consider the sequence of integers $\B=(B_n)_{n\in\NN}$ with $ B_i=2^ii! $. This means that
$$ \B=(1,2,8,48,\cdots)\; . $$
Let us now take an integer $ N>0 $, it is obvious that
\[ B_{n}\leqslant N<B_{n+1}\;  \text{ for some }n \in \NN\; . \]\label{Npositif}
By dividing $ N $ by $ B_n $, one obtains
	\begin{equation}\label{r1}
	N=d_{n} B_{n}+R_1\;  \text{ with }\begin{cases}
	\; 0\leqslant R_1<B_n\\ \;  1\leqslant d_n\leqslant 2n+1
	\end{cases} .
	\end{equation}
Here is why $  1\leqslant d_n\leqslant 2n+1 $.
\begin{proof}
	We have on the one hand $ d_n \geqslant 1 $ because $ N\geqslant B_n $, on the other hand assume that  
\begin{eqnarray*}
 d_n\geqslant 2(n+1),\text{it follows that} \\  \;  N\geqslant d_n B_n \geqslant 2(n+1) 2^n n!=B_{n+1}\;  
\end{eqnarray*}
 which is in contradiction with $ N<B_{n+1} $,
 so we have $ d_n <2(n+1) $.
\end{proof}
Concerning the remainder $ R_1 $, one distinguishes two cases : $ R_1 \neq 0 $ and $ R_1 = 0 $. 
For  $ R_1 \neq 0 $, we substitute $ N $ by it then, we  divide it by $ B_{n-1} $.
In order to allow us to repeat the same operation on the next remainders, let us assume that each time we divide $ R_i $ by $ B_{n-i} $ we obtain $ R_{i+1} \neq 0 $ as remainder. This means that first of all, we have
\begin{equation}\label{r2}
R_1=d_{n-1} B_{n-1}+R_2\;  \text{ with }\begin{cases}
\; 0< R_2<B_{n-1}\\ \;  0\leqslant d_{n-1}< 2n
\end{cases} .
\end{equation}
We have seen that $ d_n \neq 0 $ but here $ d_{n-1} $ may be zero.
According to our hypothesis $ R_1 \neq0 $, we just know from equation (\ref{r1}) that $ 0<R_1<B_n $. That's why, one or other of the following cases appears :
\begin{itemize}
\item	 $ R_1 <B_{n-1} $ which gives $ d_{n-1}=0 $
\item $ R_1\geqslant B_{n-1} $, which gives $ d_{n-1}\geqslant 1 $.
\end{itemize}
Thus we obtain the first inequality $  0\leqslant d_{n-1} $. As $ R_1 < B_n $, we have $ d_{n-1}<2n $. Therefore, $ 0\leqslant d_{n-1}< 2n $.
From equations (\ref{r1}) and (\ref{r2}), we write
$$ N=d_n B_n +d_{n-1} B_{n-1}+R_2\; . $$
By continuing in this way, we have for $ i=2,3 $ and so on
\begin{equation*}
	R_i=d_{n-i} B_{n-i}+R_{i+1}\;  \text{ with }\begin{cases}
		\; 0< R_{i+1}<B_{n-i}\\ \;  0\leqslant d_{n-i}\leqslant 2(n-i)+1
	\end{cases} .
\end{equation*}
At last, the integer $ N $ may be written in the form
\begin{equation}\label{r3}
N=d_n B_n +\cdots +d_1 B_1+d_0 B_0\; \text{ where }\begin{cases}
\;  d_i \in \{0,1,2,\cdots,2i+1 \}\\ \;  B_i =2^i i!
\end{cases} . 
\end{equation}By convention, we denote this integer $ N $ by the representation
$$ d_n :d_{n-1}:d_{n-2}:\; \cdots \; :d_2 :d_1 :d_0\;  \text{ where the }d_i\text{'s are digits.} $$
Let us now deal with $ R_1 =0 $. Throughout the successive divisions, if one of the obtained remainders is zero, then from this remainder, all the digits $ d_i $'s will be zero. Let $ R_k =0 $ be this remainder 
, that is $ R_{k-1}=d_{n-k+1} B_{n-k+1}+0 $ and the notation of the integer $ N $ will be
$$ d_n :d_{n-1}: \; \cdots \; :d_{n-k+1}:\underbrace{\;  0:\;  \cdots \; : 0 \;  }_{n-k+1 \text{ times}}\; . $$
For instance, if $ N=d_n B_n +0 $, i.e $ R_1 =0 $, we write $ N=d_n :\underbrace{\;  0:\;  \cdots \; : 0 \;  }_{n \text{ times}}$.

Like this, we have just written an integer $ N>0 $ in a special enumeration system.
\begin{defn}\label{classical}
Hyperoctahedral number system is a system that expresses all natural number $ n$ of $ \NN $ in the form: 
\begin{equation}
n=\sum_{i=0}^{k(n)} n_i. B_i \; , \text{ where }k(n)\in\NN,\;  n_i \in \{0,1,2,\cdots,2i+1 \}\; \text{ and }\; B_i =2^i i!\;   .
\end{equation}
\end{defn}
This definition is motivated by the fact that we have taken as basis $\B=(B_0,B_1,B_2,B_3,\cdots)$ where $ B_i $ is the cardinal of the hyperoctahedral group $ \B _i $. 
\begin{defn}
Let	$ n= d_{k-1}:d_{k-2}: \cdots :d_1 :d_0 $ be a number in hyperoctahedral system. We say that $ n $ is a $ k $-digits number if the first digit $ d_{k-1} $ is not zero.
\end{defn}
From equation (\ref{r3}) with the hypothesis $  B_{n}\leqslant N<B_{n+1} $ on page \pageref{Npositif}, we see that any number between $ B_k $ and $ B_{k+1} $ is a $ (k+1) $-digits number to the base $\B=(B_i)_{i\in\NN}$.
\subsection*{Converting }
Let us convert an integer from decimal system to hyperoctahedral system by means of Horner's procedure. Actually it is a scheme based on expressing a polynomial by a particular expression. For instance:
$$
a + bx  +c x^2 d x^3 + e x^4 \;=\; a + x \cdot\left(b + x \cdot(c + x \cdot(d + x \cdot e))\right). 
$$

To express a positive integer $n$ in the hyperoctahedral system, one proceeds with the following manner. Start by dividing $n$ by $2$ and let $d_0$  be the rest $r_0 $  of the expression
$$n=r_0+(2)q_0 \; .$$
Divide $q_0$ by $4$, and let $d_1$  be the rest $r_1 $ of the expression
$$q_0=r_1+(4)q_1 \; .$$ 
Continue the procedure by dividing $q_{i-1}$ by $2(i+1) $ and taking $d_i:=r_i$ of the expression
$$q_{i-1}=r_i+2(i +1)q_i \; $$
until $ q_l =0 $ for some $ l\in \NN $. In this way, we obtain $ n= d_l:d_{l-1}: \cdots :d_1 :d_0 $ and we also have
\begin{equation*}
n= d_0 + 2 \left(d_1 + 4 \cdot(d_2 + 2(3) \cdot(d_3 + \cdots))\right) .
\end{equation*}
Now let	$ n= d_{k-1}:d_{k-2}: \cdots :d_1 :d_0 $ be a number in hyperoctahedral system. By definition \ref{classical}, one way to convert $ n $ to the usual decimal system is to calculate
$$ d_{k-1} 2^{k-1}(k-1)! +\cdots +d_1 .2+d_0 \; . $$
In practice, one can use this algorithm :
\begin{center}
	\begin{tabular}{c}
		\hline\\
		\begin{tabular}{lc}
			\verb|Input :|
			& \verb|An integer| $ d_{k-1}:d_{k-2}: \cdots : d_1 :d_0 $ \verb|in hyperoctahedral system.|\\
			\verb|Output :| & \verb|An integer| $ d $ \verb|in decimal system.|\\
			\hline\\
		\end{tabular}\\
		\begin{tabular}{ll}
			1. \textbf{initiate the value of $ d $ }: &  $ d  \leftarrow d_{k-1} $\\
2. \textbf{for} $ i $ from $ k-1 $ to 1 do : &  $ d \leftarrow  d.2.i+d_{i-1} $\\
			3. \textbf{return} $ d $& 
		\end{tabular}\\
		\hline
	\end{tabular}
\end{center}\vspace{3mm}
\begin{ex}
To convert the number $ 7:0:2:3:1 $ to the decimal system, multiply 1, 3, 2, 0, 7 respectively by $ B_0 $, $ B_1 $, $ B_2 $, $ B_3 $, $ B_4 $, after that, add the results :
$$
7.384+0+2.8+3.2+1=2711\; .
$$
We obtain the same result with :
\begin{eqnarray*}
d_4=& 7 \; ,& \\
7.2.4+d_3 =& 56+0 =& 56 \; ,\\
56.2.3+d_2=& 336+2=& 338 \; ,\\
338.2.2+d_1=& 1352+3=& 1355\; ,\\
1355.2.1+d_0=& 2710+1=& 2711 \; .
\end{eqnarray*}
Let us now convert $ 2711 $ to the hyperoctahedral system by dividing it by $ 2 $, the obtained quotient by $ 4 $, and so on until we have zero as quotient. The digits that we search are the successive remainders. We shall find $ 7:0:2:3:1 $.
$$
\begin{array}[t]{rrrrrl}
2711~ \vline& \hspace{-5mm}\underline{~~~~2~~~}~& & & & \\
11\; \; \; \vline & 1355~ \vline& \hspace{-5mm}\underline{\; \; ~4~~~}~& & & \\
11~ \vline& 15\; \; \; \vline & 338~ \vline & \hspace{-5mm}\underline{\; \; \; 6~~}~& & \\
1~ \vline& 35~ \vline & 38~ \vline & 56~ \vline & \hspace{-5mm}\underline{~~8~}~& \\
& 3~ \vline& 2~ \vline& 0~ \vline & 7~ \vline & \hspace{-3mm}\underline{~10~} \\
 & & & & 7~ \vline & 0
\end{array}
$$
\end{ex}
The first ninety numbers written in hyperoctahedral system are given in table \ref{table1}.
\begin{table}
\begin{center}
\begin{tabular}{|| c | c || c | c || c | c ||}
			\hline
			Decimal & Hyperoctahedral  & Decimal & Hyperoctahedral  & Decimal &  Hyperoctahedral \\ \hline
			0 & 0  & 30 & 330 & 60 & 1120    \\ \hline
			1 & 1  & 31 & 331 & 61 & 1121 \\ \hline
			2 & 10 & 32 & 400 & 62 & 1130  \\ \hline
			3 & 11 & 33 & 401 & 63 & 1131   \\ \hline
			4 & 20 & 34 & 410 & 64 & 1200   \\ \hline
			5 & 21 & 35 & 411 & 65 & 1201   \\ \hline
			6 & 30 & 36 & 420 & 66 & 1210   \\ \hline
			7 & 31 & 37 & 421 & 67 & 1211  \\ \hline
			8 & 100 & 38 & 430 & 68 & 1220  \\ \hline
			9 & 101 & 39 & 431 & 69 & 1221  \\ \hline
			10 & 110 & 40 & 500 & 70 & 1230 \\ \hline
			11 & 111 & 41 & 501 & 71 & 1231 \\ \hline
			12 & 120 & 42 & 510 & 72 & 1300 \\ \hline
			13 & 121 & 43 & 511 & 73 & 1301 \\ \hline
			14 & 130 & 44 & 520 & 74 & 1310 \\ \hline
			15 & 131 & 45  & 521 & 75 & 1311 \\ \hline
			16 & 200 & 46 & 530 & 76 &  1320 \\ \hline
			17 & 201 & 47 & 531 & 77 & 1321 \\ \hline
			18 & 210 & 48 & 1000 & 78 & 1330\\ \hline
			19 & 211 & 49 & 1001 & 79 & 1331 \\ \hline
			20 & 220 & 50 & 1010 & 80 & 1400 \\ \hline
			21 & 221 & 51 & 1011 & 81 & 1401 \\ \hline
			22 & 230 & 52 & 1020 & 82 & 1410 \\ \hline
			23 & 231 & 53 & 1021 & 83 & 1411 \\ \hline
			24 & 300 & 54 & 1030 & 84 & 1420 \\ \hline
			25 & 301 & 55 & 1031 & 85 & 1421 \\ \hline
			26 & 310 & 56 & 1100 & 86 & 1430 \\ \hline
			27 & 311 & 57 & 1101 & 87 & 1431 \\ \hline
			28 & 320 & 58 & 1110 & 88 & 1500 \\ \hline
			29 & 321 & 59 & 1111 & 89 & 1501 \\ \hline
		\end{tabular}
\end{center}		
\caption{Representing positive integers in the decimal system and in the hyperoctahedral system.}
\label{table1}
\end{table}
\section{Signed permutations and their classification}\label{sectionClassification}
Let $ n\in\NN $ and $ \pi\in\B_n $.
We code the signed permutation $ \pi $ by $ \mathtt{inv}_1 \pi:\cdots : \mathtt{inv}_n\pi $.
\begin{lem}\label{lem1}
	let $ i,j \in [n] $ and $ \pi \in\B_n $. If
\begin{enumerate}
	\item[(i)] $ \pi(i)=j $, then
$$
\mathtt{inv}_i\,\pi = \# \{k \in \{i+1,\ldots , n\} \ |\ j>|\pi(k)|\},
$$	
\item[(ii)] $ \pi(i)=-j $, then
$$
\mathtt{inv}_i\,\pi = 1+\# \{k \in \{i+1,\ldots , n\} \ |\ j>|\pi(k)|\}+2.\# \{k \in \{i+1,\ldots , n\} \ |\ j<|\pi(k)|\}.
$$	
\end{enumerate}	
\end{lem}
\begin{proof}
	Use the definition of the number of $ i $-inversions (equations (\ref{eq1}) and (\ref{eq2})).
\end{proof}
\begin{lem}\label{lem3}
Let $ i \in [n] $ and $ \pi \in\B_n $. We have $ \mathtt{inv}_i\pi \in\{0,1,\ldots,2(n-i)+1\} $.
\end{lem}		
\begin{proof}
We deduce that from Lemma \ref{lem1}.
\end{proof}
From lemma \ref{lem3} we see that : 
\begin{eqnarray*}
&&\mathtt{inv}_1\pi \in\{0,1,\ldots,2n-1\},\\
&&\mathtt{inv}_2\pi \in\{0,1,\ldots,2(n-2)+1\},\\
&&\vdots \\
&&\mathtt{inv}_{n-1}\pi \in\{0,1,2,3\},\\
&&\mathtt{inv}_n\pi \in\{0,1\}.
\end{eqnarray*}

In other words,  the code $ \; \mathtt{inv}_1 \pi:\cdots : \mathtt{inv}_n\pi\;  $ has the same property than a $ n $-digits number in hyperoctahedral system. When we arrange in order all the elements of the hyperoctahedral group $ \B_n $, then the rank of $ \pi $ is $ 1+p $ where $ \; \mathtt{inv}_1 \pi:\cdots : \mathtt{inv}_n\pi\;  $ represents the number $ p $ in hyperoctahedral system.
\begin{ex}
	Let us consider the signed permutation $ 	\pi = \left( \begin{array}{crcc} 1 & 2 & 3 & 4\\ 
	1 & -3  & 4 & 2 \end{array} \right) \; 
	 $.
From equations (\ref{eq1}) and (\ref{eq2}), 	 
$$ \mathtt{inv}_1 \pi:\mathtt{inv}_2 \pi :\mathtt{inv}_3 \pi :\mathtt{inv}_4\pi\; =\; 0:4:1:0 $$
which is the representation of $ 0.48+4.8+1.2+1.0=34 $ in hyperoctahedral system. The rank of $ \pi $ is $ 34+1=35 $ in $ \B_4 $.
\end{ex}
Given a signed permutation of the elements of $ [n],\; n>0 $, we have just seen a kind of classification with which we determine the rank of this permutation. Now, considering the rank $ k $ of a signed permutation, we want to generate the $ k $-th signed permutation of $ \B_n $. An efficient way to derive such signed permutation is to first convert $ k-1 $ in hyperoctahedral system and then use the result to compute the corresponding permutation. Actually, each number in hyperoctahedral system determines an unique signed permutation.

Let us denote $ k-1 $ by $ \; \gamma_{n-1}:\cdots:\gamma_0\;  $ in hyperoctahedral system. Recall that we search the corresponding signed permutation of rank $ k $. As the numbers of $ i $-inversions has exactly the same property than the digits in hyperoctahedral system, we are going to generate a permutation $ \pi $ such that 
$$ \mathtt{inv}_1 \pi:\cdots : \mathtt{inv}_n\pi =\gamma_{n-1}:\cdots:\gamma_0 \; . $$
We start by defining for $ \ell\in \NN $ and $ n>0 $ the following mapping : 
\begin{equation*}
\begin{array}[t]{rccl}
\M_\ell : & [2n-1]\cup\{0\} & \longrightarrow
& [n-1]\cup\{0\}\times\{-1,1\} \\
& \gamma& \longmapsto& \begin{cases}
(\gamma ,1)& \text{ if }\gamma\leq \ell\\
(1+2\ell-\gamma ,-1)&\text{ if }\gamma>\ell
\end{cases} \\
\end{array}.
\end{equation*}
We have already seen that 
\begin{equation*}
\left( \begin{array}{cccc} 1 & 2 & \dots & n\\ 
\varepsilon_1 \sigma_1 & \varepsilon_2 \sigma_2 & \dots & \varepsilon_n\sigma_n \end{array} \right) \; 
\text{ with } \sigma\in \mathcal{S}_n \text{ and } \varepsilon_i \in \{\pm 1\}
\end{equation*}
denote a signed permutation of elements of $ [n] $.
Thus finding $ \pi $ comes back to find  a permutation $  \sigma$ of $ \mathcal{S}_n \text{ and } \varepsilon_i \in \{\pm 1\} $. Considering $ \M_i(\gamma_i)=(m_i,\epsilon_i) $ for $ i\in[n-1]\cup\{0\} $, we obtain the two sequences :
$$
\epsilon =(\epsilon_{n-1},\ldots,\epsilon_0)
\text{ and }m=(m_{n-1},\ldots,m_0)\; .
$$
Taking $ \varepsilon_i =\epsilon_{n-i} $, we obtain the $ \varepsilon_i $'s. The next step to do is finding $ \sigma\in\S_n $. By the definition of the mapping $ \M_\ell $, we deduce that $ m_i\in\{0,1,2,3,\ldots,i\} $. Thereby
\begin{eqnarray*}
&& m_{n-1}\in\{0,1,2,3,\ldots,n-1\},\\
&&\vdots \\
&& m_1\in\{0,1\} , \\
&& m_0=0 \ .
\end{eqnarray*}
For an element of $ \S_n $, the number of inversions between an object and those after the latter only varies from zero to $ p $, where $ p $ indicates the number of the following objects. The $ m_i $'s have the same property than this number of inversions. Therefore, we are going to search $ \sigma\in\S_n $ which verifies 
$$
\mathsf{inv}_1 \sigma\;  \cdots\;  \mathsf{inv}_n \sigma =m_{n-1}\; \cdots\; m_0 \text{ where }\mathsf{inv}_i \sigma =\#\{i<j<n \ |\ \sigma(i)>\sigma(j)\}\; .
$$
Let $ r_i=1+m_{i-1} $.

$ \sigma_1 $ is the $ r_n $-th element of the list :  $ 1, 2, 3, \ldots ,n $ and then one deletes it from the list.
$ \sigma_2 $ is the $ r_{n-1} $-th element among the rest of the list and one also deletes it from this one. And so on $ \sigma_n $ is the unique element of the last list. This procedure can be found for instance in the work of Laisant\cite{Laisant}.
\begin{ex}
One asks the 35-th signed permutation of the elements of the set $ [4] $. In hyperoctahedral system, we represent $ 35-1=34 $ by $ 0:4:1:0 $ (we take four digits because the set $ [4] $ has four elements). We use the mappings $ \M_1, \ldots,\M_4 $ to obtain the two sequences :
$$
\epsilon=(1,-1,1,1) \text{ and }m=(0,1,1,0)\; .	
$$
Adding an unit to each element of $ m=(0,1,1,0) $ gives the ranks $ r_4=1,r_3=2,r_2=2,r_1=1 $. Thereby among the elements of the set $ [4] $ : $ 1,2,3,4 $ written in their order, one takes the one of rank $ r_4=1 $, that is $ 1 $, then the second among $ 2,3,4 $ which is $ 3 $, next one takes $ 4 $ or the one of rank $ r_2 $ among $ 2,4 $, at last the first of the list which is $ 2 $. Thus one has
$$ \sigma=1342\in\S_4\; . $$
From the sequence $ \epsilon $, one forms the thirty fifth signed permutation of the hyperoctahedral group $ \B_4 $ :
$$
\left( \begin{array}{cccc} 1 & 2 & 3 & 4\\ 
(1)\sigma_1 & (-1)\sigma_2  & (1)\sigma_3 & (1)\sigma_4 \end{array} \right) =\left( \begin{array}{crcc} 1 & 2 & 3 & 4\\ 
1 & -3  & 4 & 2 \end{array} \right) \; .
$$ 
\end{ex}

\end{document}